\documentclass[reqno,a4paper,11pt]{article}
\usepackage{amsmath,amsthm,dsfont,amsfonts,amssymb,fancyhdr}
\usepackage{enumerate}
\usepackage[usenames,dvipsnames]{color}
\usepackage{bbm,soul,supertabular,longtable,verbatim,extarrows}
\usepackage{titlesec}
\usepackage{graphicx}
\usepackage{BOONDOX-cal}
\usepackage{cite}
\usepackage{tikz}
\usepackage{booktabs,multirow,makecell}
\usepackage{authblk}
\usetikzlibrary[trees]

\usepackage{lipsum}
\usepackage{mathrsfs}
\usepackage{indentfirst}
\usepackage[colorlinks=true, allcolors=blue]{hyperref}
\usepackage[top=2.4cm,bottom=2.2cm,left=2.6cm,right=2cm]{geometry}
\usepackage{cancel}

\usepackage{lmodern}
\usepackage[T1]{fontenc}
\usepackage[utf8]{inputenc}

\newtheorem{theorem}{Theorem}[section]
\newtheorem{proposition}[theorem]{Proposition}
\newtheorem{lemma}[theorem]{Lemma}
\newtheorem{corollary}[theorem]{Corollary}

\theoremstyle{definition}

\newtheorem{remark}[theorem]{Remark}

\soulregister\cite7
\soulregister\citep7
\soulregister\citet7
\soulregister\ref7
\soulregister\pageref7

\begin{document}
\title{\textbf{Distance-regular graphs with a few $q$-distance eigenvalues}}

\author[a]{Mamoon Abdullah}

\author[a,$_{\footnote{Corresponding author}}$]{Brhane Gebremichel}

\author[b]{Sakander Hayat}

\author[a,c]{Jack H. Koolen}

\affil[a]{\emph{School of Mathematical Sciences,
University of Science and Technology of China,
Hefei, Anhui, 230026, PR China.}}

\affil[b]{\emph{Faculty of Science, Universiti Brunei Darussalam,
Jln Tungku Link, Gadong BE1410, Brunei Darussalam.}}

\affil[c]{\emph{CAS Wu Wen-Tsun Key Laboratory of Mathematics,
University of Science and Technology of China,
Hefei, Anhui, 230026, PR China.}}

\date{}
\maketitle
\newcommand\blfootnote[1]{%
\begingroup
\renewcommand\thefootnote{}\footnote{#1}%
\addtocounter{footnote}{-1}%
\endgroup}
\blfootnote{E-mail addresses: {\tt MamoonAbdullah@hotmail.com} (M. Abdullah), {\tt brhane@ustc.edu.cn} (B. Gebremichel), {\tt sakander1566@gmail.com} (S. Hayat), {\tt koolen@ustc.edu.cn} (J. H. Koolen).}

\begin{abstract}
 In this paper we study when the $q$-distance matrix of a distance-regular graph has few distinct eigenvalues. We mainly concentrate on diameter $3$.
\end{abstract}



\section{Introduction}

In this paper all the graphs are finite, undirected and simple.
For undefined notations and terminologies, see \cite{BH12}.
Let $G$ be a connected graph. For $q\neq 0$, define the $q$-distance matrix $\mathbb{D}_q(G)$ by $(\mathbb{D}_q(G))_{xy} = 1 + \frac{1}{q} + \cdots + \frac{1}{q^{d-1}}$ if $x,y \in V(G)$ and $d(x,y)=d \geq 1$ and $(\mathbb{D}_q(G))_{xx} =0$ for all $x \in V(G)$. In \cite{KMGH2023}, it was shown that for a distance-regular graph $G$ with classical parameters $(D,b,\alpha, \beta)$, its $b$-distance-matrix has exactly $3$ distinct eigenvalues of which one is equal to $0$. Examples of distance-regular graphs with classical parameters are
the Johnson graphs, Hamming graphs, Grassmann graphs, bilinear forms graphs, among others. For more information about these graphs, we refer to \cite[Chapters 6 and 9]{BCN}.

Many people have studied when the distance matrix (that is the $1$-distance matrix) has a small number of distinct eigenvalues, see \cite{AAB2016,AAKS2017,AP2015, AP2015-2,HHL2018,KHI2016,KS1994,LHWS2013,Zhang2021,ZL2023}.

In this paper, we study distance-regular graphs with diameter $3$ whose $\mathbb{D}_q$-matrix has at most $3$ distinct eigenvalues. When $q > 0$ or $q \leqslant -1$, the $q$-distance matrix $\mathbb{D}_q$ has at least $3$ distinct eigenvalues, and we give sufficient and necessary conditions (see Corollary $3.6.$) for when $\mathbb{D}_q$ has exactly $3$ distinct eigenvalues. By applying our result to $q=1$, this solves Problem $4.1$ and Problem $4.2$ of Atik and Panigrahi \cite{AP2015} for the class of distance-regular graphs with diameter $3$. Also, we give a new infinite family of distance-regular graphs with diameter $3$ such that the distance matrix has exactly $3$ distinct distance eigenvalues.
Furthermore, we study distance-regular graph $G$ with diameter $3$ such that $\mathbb{D}_q(G)$ has exactly $2$ distinct eigenvalues. Moreover, when $G$ is antipodal we give a classification. In the last section, we show with a new proof that an antipodal distance-regular graph with diameter $D$ has quite fewer than $D$ distinct distance eigenvalues, a result that was first shown by Atik and Panigrahi \cite{AP2015}.


\section{Preliminaries}
A \emph{graph} $G$ is a pair $(V(G),E(G))$, where $V(G)$ is called the \emph{vertex set} and $E(G)\subseteq{V(G)\choose2}$
is called the \emph{edge set} of $G$. Two vertices $u,v\in V(G)$ are said to be \emph{adjacent}
if $uv\in E(G)$ and we denote this by $x \sim y$. The valency $d_u$ of a vertex $u\in V(G)$ is defined as $d_u=|\{v\in V(G):uv\in E(G)\}|$. A graph is said to be \emph{$k$-regular}, if $d_u=k$ for each $u\in V(G)$. A graph is called \emph{regular} if it is $k$-regular for some $k$.
Let $G$ be a connected graph. The {\em distance} between two vertices $u,v\in V(G)$ in graph $G$, denoted by $d(x,y)$, is the length of a shortest path between $u$ and $v$.
The {\em diameter} $D$ of $G$ is defined as $D=\max\{d(u,v):u,v\in V(G)\}$. A graph is called \emph{bipartite} if its vertex set can be partitioned into two parts such that every edge in the graph has one end (vertex) in each part. A \emph{subgraph} of a graph $G$ is a graph $H$ such that $V(H)\subseteq V(G)$ and $E(H)\subseteq E(G)$. If $E(H)=\binom{V(H)}{2}\cap E(G)$, then we say $H$ is an \emph{induced subgraph} of $G$.

A real $(n \times n)$-matrix $T$ with only non-negative entries, is called \emph{irreducible} if for all $1 \leq i, j \leq n$, there exists $k >0$ such that $(T^k)_{ij} >0$.
If $M$ (resp. $N$) is a real symmetric $m\times m$ (resp. $n\times n$) matrix, let $\eta_1(M)\geqslant\eta_2(M)\geqslant\cdots\geqslant\eta_m(M)$ (resp. $\eta_1(N)\geqslant\eta_2(N)\geqslant\cdots\geqslant\eta_n(N)$) denote the eigenvalues of $M$ (resp. $N$) in nonincreasing order. Assume $m\leqslant n$. We say that the eigenvalues of $M$ \emph{interlace} the eigenvalues of $N$, if $$\eta_{n-m+i}(N)\leqslant\eta_i(M)\leqslant\eta_i(N)$$
for each $i=1,\ldots,m$. The following result is a special case of interlacing.

\begin{lemma}[{Cf.\cite[Theorem 9.1.1]{GD01}}]\label{interlacing}
Let $B$ be a real symmetric $n\times n$ matrix and $C$ a principal submatrix of $B$ of order $m$, where $m<n$. Then the eigenvalues of $C$ interlace the eigenvalues of $B$.
\end{lemma}


Let $G$ be a connected regular graph with diameter $D$. For $1\le i\le D$, we define $G_{i}(u):=\{v\in G\mid d(u,v)=i\}$. We write $G(u)=G_{1}(u)$.
The local graph of $G$ with respect to $u\in V(G)$ is the subgraph of $G$, induced by $G(u)$.
The {\em distance-}$i$ graph of $G$ is the graph induced by $G_i$, for $i=0,1,\ldots,D$.
The graph $G$ is called {\em distance-regular}, if there are numbers $a_i,b_i,c_i~(0\leq i\leq D)$, such that if $d(x,y)=i$, then
\begin{equation*}
\mid G_{i-1}(x)\cap G(y)\mid =c_{i},~~\mid G_{i}(x)\cap G(y) \mid =a_{i},~~\mid G_{i+1}(x)\cap G(y)\mid =b_{i}.
\end{equation*}
Note that $G$ is $b_0$-regular. It follows that, $a_i + b_i + c_i =b_0$ and $c_0=b_D=0$. Usually, we write $k = b_0$. 
The sequence
\begin{equation}\label{intersection array}
  \{ b_0, b_1, \ldots, b_{D-1} ; c_1=1, c_2, \ldots ,c_D \}
\end{equation}
is called the intersection array of the distance-regular graph $G$.

For a distance-regular graph $G$, the adjacency matrix $A_i$ of $G_i$ is called the {\em distance-$i$ matrix} of $G$, for $i=0,1,\ldots,D$.
The matrices $A_i$'s satisfy the following relations {\cite[Section 4.1]{BCN}}
\begin{equation}\label{recurrence relation}
  A_0=I, \text{ } A_1=A, \text{ } AA_i = c_{i+1}A_{i+1}  +a_{i}A_{i} +b_{i-1}A_{i-1}, \text{ for } i= 0, 1, \ldots , D,
\end{equation}
where $A_{-1} =A_{D+1} = \textbf{0}$, while the numbers $b_{-1}$ and $c_{D +1}$ are unspecified.
Furthermore $A_0 + A_1 + \cdots +A_D= J$ holds.

A distance-regular graph $G$ is said to be \emph{antipodal}, whenever for any vertices $u,v,w$ such
that $d(u,v)=d(u,w)=D$ implies that $d(v,w)=D$ or $v=w$, that is, the relation of being at distance $D$ or
zero is an equivalence relation on $V(G)$, and the equivalence classes are called \emph{antipodal classes} or \emph{fibers}.
We say that $G$ is an antipodal $r$-cover if the equivalence classes have size $r$. Moreover, when $G$ is antipodal of diameter $D$ then we can define a new graph $\tilde{G}$, the \emph{folded graph} of $G$, which has the antipodal classes of $G$ as vertices, and two antipodal classes $C_1$ and $C_2$ are adjacent if there exist $x \in V(C_1)$ and $y \in V(C_2)$ that are adjacent.

From Equation (\ref{recurrence relation}), we have that for all $i$ ($0 \leqslant i \leqslant D$)
the distance-$i$ matrix $A_i$ can be written as a polynomial of degree $i$ in $A$. 
Moreover, as the $q$-distance matrix $\mathbb{D}_q$ is
$\mathbb{D}_q = \sum\limits_{i=1}^D (1 +\frac{1}{q} + \cdots + \frac{1}{q^{i-1}})A_i$,
it follows that $\mathbb{D}_q$ can be written as a polynomial of degree $D$ in $A$,
so $\mathbb{D}_q(G) = R_q(A)$, and for every eigenvalue $\theta$ of $G$, $R_q(\theta)$
is its $\mathbb{D}_q$-eigenvalue. This, in particular, implies that $\mathbb{D}_q$ has at most $D+1$ distinct eigenvalues.  We say that the eigenvalues of $\mathbb{D}_q$ are the \emph{$q$-distance eigenvalues} of $G$. In this paper, we are interested to see when
$\mathbb{D}_q$ has less number of distinct eigenvalues.

\section{Eigenvalues of a $q$-distance matrix of distance-regular graphs of diameter 3}
In this section, we will study the $q$-distance eigenvalues of a distance-regular graph $G$ of diameter $3$. We will give sufficient conditions when $G$ has at most 3 distinct
$q$-distance eigenvalues.

\begin{proposition}\label{D_q eigenvalues d=3 degree 3}
Let $G$ be a distance-regular graph with $n$ vertices, diameter $3$, intersection array $\{k, b_1, b_2 ; 1, c_2, c_3\}$ and distinct (adjacency) eigenvalues $k= \theta_0 > \theta_1 >\theta_2 > \theta_3$. Then for $q \neq 0$, the $\mathbb{D}_q$-eigenvalues of $G$ are:

\begin{equation*}
\begin{split}
R_q(\theta_i) =  \frac{1}{c_2c_3q^2}(((q^2 +q +1)a_2 -(q^2+q)c_3)k + (q^2c_2c_3 -(q^2 +q)a_1c_3 + (q^2+q+1)(a_1a_2 -b_1c_2 -k))\theta_i \\
+((q^2 +q)c_3 -(q^2+q+1)(a_1+a_2))\theta_i^2 +(q^2+q+1)\theta_i^3), \text{ where } 0 \leqslant i \leqslant 3.
\end{split}
\end{equation*}

\end{proposition}

\begin{proof}
Let $G$ be a distance-regular graph on $n$ vertices of diameter $3$. From Equation (\ref{recurrence relation}), we have that $A^2 =kI_n +a_1A +c_2A_2$ and $AA_2 =b_1A +a_2A_2 +c_3A_3$, where $I$ is the identity matrix of order $n$. Thus, $c_2A_2 = A^2 -a_1A- kI $ and $$c_2c_3 A_3 = ka_2I +(a_1a_2 -b_1c_2 -k)A - (a_1 +a_2)A^2 +A^3.$$

Since
\begin{equation*}
\mathbb{D}_q(G) = A + (1 +\frac{1}{q})A_2 + (1 +\frac{1}{q} + \frac{1}{q^2})A_3,
\end{equation*}
we have
\begin{equation*}
\begin{split}
c_2c_3q^2 \mathbb{D}_q(G)  = ((q^2 +q +1)a_2 -(q^2+q)c_3)kI + (q^2c_2c_3 -(q^2 +q)a_1c_3 + (q^2+q+1)(a_1a_2 -b_1c_2 -k))A \\
+((q^2 +q)c_3 -(q^2+q+1)(a_1+a_2))A^2 +(q^2+q+1)A^3.
\end{split}
\end{equation*}
Thus, if $\theta_i$ is an eigenvalue of $G$, then the $\mathbb{D}_q$-eigenvalues of $G$ is
\begin{equation*}
\begin{split}
R_q(\theta_i) =  \frac{1}{c_2c_3q^2}(((q^2 +q +1)a_2 -(q^2+q)c_3)k + (q^2c_2c_3 -(q^2 +q)a_1c_3 + (q^2+q+1)(a_1a_2 -b_1c_2 -k))\theta_i \\
+((q^2 +q)c_3 -(q^2+q+1)(a_1+a_2))\theta_i^2 +(q^2+q+1)\theta_i^3).
\end{split}
\end{equation*}

\end{proof}

Since the adjacency matrix of $G$ has constant row-sum and is symmetric, the eigenvalues, that are not the valency, have an eigenvector orthogonal to the all-ones vector. We obtain the following:

\begin{proposition}\label{D_q eigenvalues d=3}
Let $G$ be a distance-regular graph with $n$ vertices, diameter $3$, intersection array $\{k, b_1, b_2 ; 1, c_2, c_3\}$ and distinct (adjacency) eigenvalues $k > \theta_1 >\theta_2 > \theta_3$. Then for $q \neq 0$ the $\mathbb{D}_q$-eigenvalues of $G$ are:
$$R_q(k) = \frac{1}{c_2q^2}(c_2(q^2 +q+1)n - (q^2 +q +1)c_2 -((q+1)c_2-a_1-1)k -k^2),$$
and
$$R_q(\theta_i) =  \frac{-1}{c_2q^2}((q^2 +q +1)c_2 -k +((q+1)c_2-a_1)\theta_i +\theta_i^2), \text{ where } 1 \leqslant i \leqslant 3.$$

\end{proposition}

\begin{proof}
Let $G$ be a distance-regular graph on $n$ vertices of diameter $3$ with distinct (adjacency) eigenvalues $k > \theta_1 >\theta_2 > \theta_3$. From Equation (\ref{recurrence relation}), we have that $A^2 =kI +a_1A +c_2A_2$, and $J =I +A + A_2 +A_3$ is the $n \times n$ all-ones matrix and $I$ is the identity matrix of order $n$. Since
\begin{equation*}
\mathbb{D}_q(G) = A + (1 +\frac{1}{q})A_2 + (1 +\frac{1}{q} + \frac{1}{q^2})A_3,
\end{equation*}
we have
\begin{equation*}
  c_2q^2\mathbb{D}_q(G) = c_2(q^2 +q+1)J - ((q^2 +q +1)c_2 -k)I -((q+1)c_2-a_1)A -A^2.
\end{equation*}

Thus the $\mathbb{D}_q$-eigenvalues of $G$ are
$$R_q(k) = \frac{1}{c_2q^2}(c_2(q^2 +q+1)n - (q^2 +q +1)c_2 -((q+1)c_2-a_1-1)k -k^2),$$
and
$$R_q(\theta_i) =  \frac{-1}{c_2q^2}((q^2 +q +1)c_2 -k +((q+1)c_2-a_1)\theta_i +\theta_i^2) \text{ for } i =1,2,3.$$

\end{proof}

Now let us study the situation $R_q(\theta_i) = R_q(\theta_j)$ for some $(i,j)$ pair, where $1 \leqslant i < j \leqslant 3$, for a distance-regular graph with diameter $3$.
\begin{proposition}\label{theta_l}
Let $G$ be a distance-regular graph with $n$ vertices, diameter $3$, intersection array $\{k, b_1, b_2 ; 1, c_2, c_3\}$ and distinct (adjacency) eigenvalues $k > \theta_1 >\theta_2 > \theta_3$. Let $1 \leqslant i < j \leqslant 3$. Then for $q \neq 0$, $R_q(\theta_i) = R_q(\theta_j)$ if and only if $\theta_l = qc_2 -b_2 +a_3$ for the unique $l \in \{1, 2, 3\}-\{i,j\}$.
\end{proposition}

\begin{proof}
Let $1 \leqslant i < j \leqslant 3$. By Proposition \ref{D_q eigenvalues d=3}, we see that, $R_q(\theta_i) = R_q(\theta_j)$ if and only if
$$((q+1)c_2-a_1)\theta_i +\theta_i^2 = ((q+1)c_2-a_1)\theta_j +\theta_j^2,$$
if and only if
$$(\theta_i - \theta_j)(\theta_i +\theta_j +(q+1)c_2) - a_1)=0.$$
That is $R_q(\theta_i) = R_q(\theta_j)$ if and only if $\theta_i +\theta_j = a_1 -(q+1)c_2$, as $\theta_i \neq \theta_j$. Note that, $k + \theta_1 + \theta_2 + \theta_3 = a_1 +a_2 +a_3$ and $k= a_2 +b_2 +c_2$. Thus, $R_q(\theta_i) = R_q(\theta_j)$ if and only if $\theta_l = qc_2 -b_2 +a_3$ for the unique $l \in \{1, 2, 3\}-\{i,j\}$.
\end{proof}

%
Note that, as the path of length $3$ (denoted by $P_4$) is an induced subgraph of a distance-regular graph $G$ of diameter $3$, we can easily see that $\theta_1(G) \geqslant \theta_1(P_4) = \frac{\sqrt{5}-1}{2} \approx 0.618$ and $\theta_3(G) \leqslant \theta_3(P_4) =\frac{-\sqrt{5}-1}{2} \approx -1.618$, by Lemma \ref{interlacing}.

\textbf{Notation}: For $l \in \{1,2,3\}$, we define $q_l = \frac{\theta_l -a_3 +b_2}{c_2}$.
Note that $q_1 > 0$ and $q_3 < 0$. Now let us count the number of distinct $\mathbb{D}_q$-eigenvalues of a distance-regular graph of diameter $3$.

The following result was shown by Koolen et al. cf. \cite[Proposition 3.2 and Proposition 3.3]{KPY2011}

\begin{proposition}\label{KPYresults}
Let $G$ be a distance-regular graph with $n$ vertices, diameter $3$ and distinct (adjacency) eigenvalues $k > \theta_1 >\theta_2 > \theta_3$. Then the following holds:
\begin{itemize}
  \item[(i)] $\theta_1 > \max \{ a_3 -b_2, 0 \}$ and $\theta_3 < \min \{ a_3 -b_2, -\sqrt{2} \}$,
  \item[(ii)] $\theta_2$ lies between $-1$ and $a_3 -b_2$,
  \item[(iii)] $\theta_l = a_3 -b_2$ if and only if $\theta_l = -1$.  If this is the case, then $l =2$.
\end{itemize}

\end{proposition}

Note that this implies that if $q_i =0$ for some $i =1,2,3$, then $i =2$ and $\theta_2 = -1$.

\begin{corollary}\label{corollary of no of non tirivial Dq evs}
Let $G$ be a distance-regular graph with $n$ vertices, diameter $3$ and distinct (adjacency) eigenvalues $k > \theta_1 >\theta_2 > \theta_3$. Then $|\{R_q(\theta_1), R_q(\theta_2), R_q(\theta_3)\}| \geq 2$, for $q \neq 0$. 
\end{corollary}

\begin{corollary}\label{number of D_q evs}
Let $G$ be a distance-regular graph with $n$ vertices, diameter $3$ and distinct (adjacency) eigenvalues $k > \theta_1 >\theta_2 > \theta_3$.
\begin{enumerate}
  \item[(i)] If $q \notin \{q_1,q_2,q_3\}$ and $q \notin (-1,0]$, then $\mathbb{D}_q$ has exactly four distinct eigenvalues.
  \item[(ii)] If $q \notin \{q_1,q_2,q_3\}$ and $q \in (-1,0)$, then $\mathbb{D}_q$ has at least three distinct eigenvalues.
  \item[(iii)] If $q \in \{q_1,q_2,q_3\}$ and $q \notin (-1,0]$, then $\mathbb{D}_q$ has exactly three distinct eigenvalues.
  \item[(iv)] If $q \in \{q_1,q_2,q_3\}$ and $q \in (-1,0)$, then $\mathbb{D}_q$ has at least two and at most three distinct eigenvalues.
\end{enumerate}
\end{corollary}

\begin{proof}
Let $G$ be a distance-regular graph of diameter $3$. For $q \notin (-1,0]$, $\mathbb{D}_q$ is a non-negative and irreducible matrix (cf. \cite[Lemma 2.6.]{KMGH2023}), so its largest eigenvalue $R_q(k)$ is simple (cf. \cite[Lemma 2.7.]{KMGH2023}). This implies, if $q \notin (-1,0]$ we have $R_q(k) \neq R_q(\theta_i)$ for $1 \leqslant i \leqslant 3$. Then, we are done by Proposition \ref{theta_l}.

\end{proof}

As a consequence of Corollary \ref{number of D_q evs}, we have the following result on the number of distinct eigenvalue of the distance matrix (that is, $q=1$) of a distance-regular graph of diameter $3$.
\begin{corollary}\label{number of D_q evs for q=1}
Let $G$ be a distance-regular graph with $n$ vertices, diameter $3$ and distinct (adjacency) eigenvalues $k > \theta_1 >\theta_2 > \theta_3$. Then $G$ has exactly 3 distinct
distance eigenvalues if and only if $\theta= c_2 + a_3 -b_2$ is a non-trivial eigenvalue of $G$.

\end{corollary}

This corollary gives necessary and sufficient conditions that the distance matrix for a distance-regular graph of diameter 3 has exactly 3 distinct eigenvalues. As the distance-matrix of a
distance-regular graph has either exactly 3 or exactly 4 distinct eigenvalues, this solves Problem 4.1
and Problem 4.2 of Atik and Panigrahi \cite{AP2015} for the class of distance-regular graphs with diameter 3.  This also improves results of Alazemi et al. \cite{AAKS2017}, who gave
partial results for diameter 3.

Next we will study cases (iii) and (iv) of Corollary \ref{number of D_q evs}, in details, for distance-regular graphs of diameter $3$ which are bipartite or antipodal.

\subsection{Distance-regular graphs with diameter $3$ and having $2$ distinct $\mathbb{D}_q$-eigenvalues}

Recall that, for distance-regular graph with diameter $3$ and distinct (adjacency) eigenvalues $k > \theta_1 > \theta_2 > \theta_3$, we have $R_q(\theta_i) = R_q(\theta_j)$ if and only if $\theta_l = qc_2 +a_3-b_2$ for $l \in \{1,2,3\}-\{i,j\}$, by Proposition \ref{theta_l}.

From Corollary \ref{number of D_q evs}(iv), we see that a distance-regular graph $G$ of diameter $3$ has at least two  and at most three distinct $\mathbb{D}_q$-eigenvalues if $q \in \{q_1,q_2,q_3\}$ and $q \in (-1,0)$. Moreover, by Proposition \ref{KPYresults}(i), we know that $\theta_1 > a_3 -b_2 > qc_2 +a_3 -b_2$, that is $R_q(\theta_2) \neq R_q(\theta_3)$ (i.e. $q \neq q_1$). Then, $G$ has $2$ distinct $\mathbb{D}_q$-eigenvalues if one of the following conditions holds:
\begin{enumerate}
  \item[(i)] $R_q(\theta_1)= R_q(\theta_2)$ and $R_q(k)= R_q(\theta_3)$,
  \item[(ii)] $R_q(\theta_1)= R_q(\theta_3)$ and $R_q(k)= R_q(\theta_2)$,
  \item[(iii)] $R_q(k)= R_q(\theta_1)= R_q(\theta_2)$, or
  \item[(iv)] $R_q(k)= R_q(\theta_1)= R_q(\theta_3)$
\end{enumerate}
%

\subsubsection{Bipartite distance-regular graphs}
Here we will show that there is no bipartite distance-regular graph of diameter $3$ whose $q$-distance matrix has exactly $2$ distinct eigenvalues.
\begin{proposition}
There exists no bipartite distance-regular graph $G$ of diameter $3$ with $2$ distinct $\mathbb{D}_q$ eigenvalues.
\end{proposition}

\begin{proof}
Note that, for a bipartite distance-regular graph $G$ of diameter $3$ with distinct (adjacency) eigenvalues $\theta_0 >\theta_1 > \theta_2 > \theta_3$, we have that $\theta_0 = k =-\theta_3$ and $\theta_1 =\sqrt{b_2} = -\theta_2$, (cf. \cite[Page 432]{BCN}).
If $-1 < q <0$, and $G$ has exactly $2$ distinct $\mathbb{D}_q$ eigenvalues, then there exists $l\in \{1,2,3\}$ such that $\theta_l = qc_2 + a_3 - b_2 = qc_2 -b_2$.
This means that $-\sqrt{b_2} > - b_2  >\theta_l > -k$, which is impossible. This shows the result.
\end{proof}

\subsubsection{Antipodal distance-regular graphs}
Here we will characterize the antipodal distance-regular graphs of diameter $3$ with exactly two $q$-distance eigenvalues.

\begin{proposition}
Let $G$ be an antipodal distance-regular graph of diameter $3$. If $G$ has exactly $2$ distinct $q$-distance eigenvalues, then $G$ is one of the following graphs:
\begin{enumerate}
  \item The Johnson graph $J(6,3)$ with intersection array $\{9, 4,1 ; 1,4,9\}$, $R_q(k)= R_q(3)=R_q(-1) =3$ for $q=\frac{-1}{2}$;
  \item The distance-2 graph of the halved $6$-cube with intersection array $\{15, 8,1 ; 1, 8,15\}$, $R_q(k)= R_q(3)=R_q(-1) =3$ for $q=\frac{-1}{2}$;
  \item The distance-2 graph of the Gosset graph with intersection array $\{27, 16, 1 ; 1, 16, 27\}$, $R_q(k)= R_q(3)=R_q(-1) =3$ for $q=\frac{-1}{2}$;
  \item A graph with intersection array $\{35,18,1; 1,18,35\}$, $R_q(k) = -28 = R_q(\theta_3)$ and $R_q(\theta_1) =8 =R_q(-1)$ for  $q = -\frac{1}{3}$.
\end{enumerate}

\end{proposition}

\begin{proof}
Note that, for an antipodal distance-regular graph $G$ of diameter $3$ with distinct (adjacency) eigenvalues $\theta_0 >\theta_1 > \theta_2 > \theta_3$, we have that $\theta_1 +\theta_3 = a_1 -c_2$ so $R_q(\theta_1) \neq R_q(\theta_3)$ for $q \neq 0$. Thus, if $G$ has $2$ distinct $\mathbb{D}_q$-eigenvalues then $q = q_3$ and one of the following holds: $R_q(k) = R_q(\theta_1)= R_q(\theta_2)$, or $R_q(\theta_1)= R_q(\theta_2)$ and $R_q(k) =R_q(\theta_3)$, with $\theta_2 =-1$, $\theta_1 =(a_1 +1) -(q +1)c_2$ and $\theta_3 =qc_2 -1$.

Let us first consider the case: $R_q(k) = R_q(\theta_1)= R_q(\theta_2)$.

For an antipodal distance-regular graph $G$ of diameter $3$ the standard sequence of the eigenvalue $k$ (resp. $-1$) is $\textbf{j}$ (resp. $\textbf{u} = (1 \text{ } \frac{-1}{k} \text{ } \frac{-1}{k} \text{ } 1)$). Then, as $\mathbb{D}_q(G) = A + (1 +\frac{1}{q})A_2 + (1 +\frac{1}{q} + \frac{1}{q^2})A_3$, we have that
$$R_q(k) = k +k_2(1+\frac{1}{q})+ k_3(1 +\frac{1}{q} +\frac{1}{q^2}), \text { and } R_q(-1) = -1 -\frac{k_2}{k}(1+\frac{1}{q})+ k_3(1 +\frac{1}{q} +\frac{1}{q^2}).$$
Now if $R_q(k) = R_q(-1)$, we have that $q = \frac{1-r}{r}$, as $k_2 = (r-1)k$ and $k_3 = r-1$. Note that  $R_q(\theta_1) =R_q(\theta_2)$ implies $\theta_3 = qc_2 -1 = (\frac{1-r}{r})c_2 -1$, and thus
$$\theta_1 =a_1 -c_2 - \theta_3 = k-(r-1)c_2 -1 -c_2 + (\frac{r-1}{r})c_2 +1 = k- (\frac{r^2-r+1}{r})c_2.$$
Moreover, as $k=-\theta_1\theta_3$ we have that $\theta_1 = \frac{r^2-r+1}{r-1} = r + \frac{1}{r-1}$.

As $\theta_1$ is an integer and $r \geqslant 2$ is an integer,  we obtain $r=2$, $q=\frac{-1}{2}$ and $\theta_1 =3$. So, if $R_q(k)=R_q(\theta_1)=R_q(\theta_2)$ then $G$ is a Taylor graph with $\theta_1 =3$ and $\theta_3= \frac{-c_2}{2}-1$. In this case the distance-2 graph of $G$ is a Taylor graph with $\theta_3 =-3$ and they are classified in \cite[Corollary 1.15.3]{BCN}. (Note that the distance-2 graph of a 3-cube is the disjoint union of two 4-cliques.)

Next we will consider the case when $R_q(\theta_1)= R_q(\theta_2)$ and $R_q(k) =R_q(\theta_3)$. As $\theta_3 = qc_2 -1$ (from $R_q(\theta_1)=R_q(\theta_2)$) and $\theta_1\theta_3 =-k$ we find
$$\theta_1 = a_1-c_2 -\theta_3= k-(r-1)c_2 -1 -c_2 -qc_2 +1 = k-(r+q)c_2.$$
As $k =-\theta_1\theta_3$ and $\theta_3 = qc_2 -1$, we have
$$\theta_1= -\theta_1 \theta_3 -(r+q)c_2 = \theta_1(1-qc_2)-(r+q)c_2.$$
This implies that $\theta_1= -1 -\frac{r}{q}$. 

The standard sequence of the eigenvalue $\theta \in \{\theta_1, \theta_3\}$ is $\textbf{u} = (1 \text{ } u_1 \text{ } \frac{-u_1}{r-1} \text{ } \frac{-1}{r-1})$ with $u_1 = \frac{\theta}{k}$. If $R_q(k) = R_q(\theta_3)$ (i.e. $ k +(r-1)k(1+\frac{1}{q})+ (r-1)(1 +\frac{1}{q} +\frac{1}{q^2}) = \frac{-\theta_3}{q} -(1 +\frac{1}{q} +\frac{1}{q^2})$), we have that
$$\frac{\theta_3}{q}  = -k(1 +(1 +\frac{1}{q})(r-1)) -r(1 +\frac{1}{q} + \frac{1}{q^2})),$$
as $u_1 =\frac{\theta_3}{k}$. Since $k = -\theta_1\theta_3$ and $\theta_1= -1 -\frac{r}{q}$, we have that
\begin{equation}\label{theta_3}
  \theta_3 = -\frac{q^2 +q +1}{q^2+q(r+1) +r-1}.
\end{equation}
As $\theta_3 <0$ and $q^2 +q +1 >0$, we find that $q^2 + q(r+1) + r-1 >0$.

Note that, for an antipodal distance-regular graph $G$ of diameter $3$ we have $\theta_1 =\sqrt{k} = -\theta_3$ or $\theta_1$ and $\theta_3$ are both integers (cf. \cite[Page 431]{BCN}). If $\theta_1 =\sqrt{k} = -\theta_3$ is not an integer, then from $\theta_1= -1 -\frac{r}{q}$, we obtain $q = \frac{r(1-\sqrt{k})}{k-1}$. Substituting $q = \frac{r(1-\sqrt{k})}{k-1}$ and $\theta_3 = -\sqrt{k}$ in  $\theta_3 = -\frac{q^2 +q +1}{q^2+q(r+1) +r-1}$ gives
\begin{equation*}
\begin{split}
 r^2(1+k) +r(k-1) +(k-1)^2 -\sqrt{k}(2r^2 + r(k-1)) = \\
  \sqrt{k}(r^2(1 +k) + r(r+1)(k-1) + (r-1)(k-1)-\sqrt{k}(2r^2 +r(r+1)(k-1))).
\end{split}
\end{equation*}
Here $\theta_3 = -\sqrt{k}$ is not an integer implies $-(2r^2 +r(k-1)) = r^2(1 +k) + r(r+1)(k-1) + (r-1)(k-1)$.  This is a contradiction as $r, k \geqslant 2$ are integers.

If $\theta_1$ and $\theta_3$ are both integers, then $\theta_1= -1 -\frac{r}{q}$ implies $q = \frac{-r}{\theta_1 +1}$ and hence $\theta_1 +1 > r$ as $q \in (-1, 0)$. Thus from Equation (\ref{theta_3}) we have
$$-\theta_3 =\frac{r^2 -r(\theta_1 +1) + (\theta +1)^2}{r^2 -r(r+1)(\theta_1 +1) +(r-1)(\theta +1)^2} = 1 + \frac{(r^2 -r +2) +\theta_1(r^2-2r +4) -\theta_1^2(r-2)}{\theta_1^2(r-1) -\theta_1(r^2 -r +2) -1}.$$
As $\theta_3(G) \leqslant \theta_3(P_4) =\frac{-\sqrt{5}-1}{2} \approx -1.618$, by Lemma \ref{interlacing} and $\theta_3(G)$ is an integer, we have $\theta_3 \leqslant -2$. This implies
\begin{equation}\label{ratio}
  \frac{(r^2 -r +2) + 4 \theta_1 +\theta_1(r-2)(r -\theta_1)}{\theta_1(r-1)(\theta_1-r)  -2\theta_1 -1}
\end{equation}
is a positive integer,  which implies $\theta_1 > r$ (as for $\theta_1 =r$ from Equation (\ref{ratio}) we obtain $\frac{r^2 +3r +2}{-2r-1}$, which is negative).

Equation (\ref{ratio}) implies that $\theta_1(r-1)(\theta_1-r)  -2\theta_1 -1$ divides $\theta_1^2 +(2-r)\theta_1 + r^2-r +2$.

If $r \geqslant 3$ and $\theta_1 \geqslant r+3$, then $\theta_1(r-1)(\theta_1-r)  -2\theta_1 -1\geqslant \theta_1^2 +(2-r)\theta_1 + r^2-r +2$, as $\theta_1(r-1)(\theta_1-r)  -2\theta_1 -1 >0$, but this is impossible. So, either $r =2$, or $\theta_1 \in \{r+1, r+2\}$.

Let us first consider the case $r =2$. Thus from Equation (\ref{ratio}) we have
\begin{equation}
\frac{4( \theta_1+1)}{\theta_1(\theta_1-4)  -1}
\end{equation}
is a positive integer and this implies that $\theta_1 \in \{3, 4, 5\}$. As $q^2 + q(r+1) + r-1 >0$ and $q = -\frac{r}{\theta_1+1}$ we find that
only $\theta_1 =5$ survives. In this case $\theta_3 = -7$ and $ q =-1/3$ and intersection array $\{35,18,1; 1,18,35\}$.

 So we are left to consider $\theta_1 \in \{ r+1, r+2\}$ and $r \geq 3$.
 For $\theta_1 = r+1$,  Equation (\ref{ratio}) becomes $\frac{2(r+2)}{r^2 -2r -4}$ is a positive integer and we obtain that the only possible value for $r$  is $r=4$.
 In this case we find $\theta_3 = -7$, $q= -2/3$ and intersection array $\{35, 27, 1; 1, 9, 35\}$. This intersection array is not feasible as $a_1k$ is not an even integer.

 For $\theta_1 = r+2,$  Equation (\ref{ratio}) becomes $\frac{-r^2+3r+2}{2r^2 -9}$ and this is only a positive integer if $r=2$. This finishes the proof.
\end{proof}

\begin{remark}
Note that, there are 3854 non-isomorphic distance-regular graphs of diameter $2$ with intersection array $\{16,9 ; 1,8\}$ (cf. \cite{MS2001}). Then, by  \cite[Theorem 1.5.3.]{BCN}, we see that there are at least $\lceil \frac{3854}{36}\rceil =108$ non-isomorphic antipodal distance-regular graphs with the intersection array $\{35,18,1; 1,18,35\}$.
\end{remark}

On this moment, the Johnson graph $J(8,3)$ with intersection array $\{15,8,3 ; 1, 4, 9\}$ is the only known primitive distance-regular graph with diameter $3$ having exactly two $q$-distance eigenvalues (we have $R_q(k)= R_q(-1) =15$ and $R_q(7) = R_q(-3) = -9$ for $q=-\frac{1}{2}$). We do not know whether there exist more such distance-regular graphs.

\subsection{Distance-regular graphs with diameter $3$ and having $3$ distinct $\mathbb{D}_q$-eigenvalues}

By Corollary \ref{number of D_q evs}(iii) we see that a distance-regular graph of diameter $3$ has exactly three distinct $\mathbb{D}_q$-eigenvalues if $q \in \{q_1,q_2,q_3\}$ and $q \notin (-1,0]$.

Note that, if $G$ is a bipartite distance-regular graph of diameter $3$ it has intersection array
$$\{ k, k-1, k-c_2; 1, c_2, k \}$$
and its distinct (adjacency) eigenvalues are $\theta_0 =k, \theta_1 = \sqrt{k-c_2}, \theta_2= -\sqrt{k-c_2}, \theta_3 = -k$. Thus for $q \notin (-1,0]$, the graph $G$ has exactly three distinct $q$-distance eigenvalues if and only if $\theta_l = (q+1)c_2 -k$ for $l \in \{1,2,3\}$, by Proposition \ref{theta_l}. That is, for $q \notin (-1,0]$, a bipartite distance-regular graph $G$ of diameter $3$ has three distinct $q$-distance eigenvalues if and only if $q = \frac{k +\theta_l}{c_2} -1$ for $l \in \{1,2,3\}$. Note that $q_3= -1$.

\vspace{5mm}
Now, let us study antipodal distance-regular graphs of diameter $3$ with exactly three distinct $\mathbb{D}_q$-eigenvalues if $q \in \{q_1,q_2,q_3\}$ and $q \notin (-1,0]$.

If $G$ is an antipodal distance-regular graph of diameter $3$ , then it has  intersection array
$$\{ k, c_2(r-1), 1; 1, c_2, k \}$$
with $c_2 < k-1$ and $r \geqslant 2$. Its distinct (adjacency) eigenvalues are $k, \theta_1, \theta_2 =-1$ and $\theta_3$, where $\theta_1$ and $\theta_3$ are the solutions of the equation $\theta^2 +(c_2 -a_1)\theta -k=0$. This implies $\theta_1 + \theta_3 =a_1 -c_2$ and $\theta_1 \theta_3 = -k$. For an antipodal distance-regular graph of diameter $3$ if its $q$-distance matrix has $3$ distinct eigenvalues, for $q \notin (-1,0]$, we give the following characterization:

\begin{proposition}\label{antipodal d=3}
Let $G$ be an antipodal distance-regular graph with diameter $3$ and distinct (adjacency) eigenvalues $k > \theta_1 > \theta_2 > \theta_3$.
\begin{enumerate}
\item For $q > 0$, $G$ has exactly three distinct $\mathbb{D}_q$-eigenvalues if and only if $\theta_3 = -\frac{r}{q}-1$.
  \item For $q \leqslant -1$, $G$ has exactly three distinct $\mathbb{D}_q$-eigenvalues if and only if $\theta_1 = -\frac{r}{q}-1$.
\end{enumerate}
\end{proposition}

\begin{proof}
Let $G$ be an antipodal distance-regular graph with diameter $3$. Let $q \notin (-1,0]$. Recall that for $q \notin (-1,0]$, $\mathbb{D}_q$ is a non-negative and irreducible matrix (cf. \cite[Lemma 2.6.]{KMGH2023}), so its largest eigenvalue $R_q(k)$ is simple. As $\theta_2 =-1$ we have $q_2 =0$ and hence $R_q(\theta_1) \neq R_q(\theta_3)$. So $G$ has three distinct $\mathbb{D}_q$-eigenvalues if and only if $R_q(\theta_2) = R_q(\theta_i)$ for $i \in \{1,3\}$ by Corollary \ref{number of D_q evs}. By Proposition \ref{D_q eigenvalues d=3} we know that $R_q(\theta_2) = R_q(\theta_i)$ if and only if  $\theta_i + \theta_2 = a_1 -(q+1)c_2$, for $i \in \{1,3\}$. Thus $G$ has three distinct $\mathbb{D}_q$-eigenvalues if and only if $\theta_i = a_1 -(q+1)c_2 +1$ for $i \in \{1,3\}$.

Let $\{i, j\} = \{1, 3\}$.  Note that $\theta_1 + \theta_3 = a_1 -c_2$. Hence $\theta_i = a_1 -(q+1)c_2 +1$ implies that $\theta_j = qc_2 -1$. Note that $\theta_1\theta_3 = -k$.
This implies that  $(a_1 -(q+1)c_2 +1)(1-qc_2) = k$. So we find $k= a_1 +1 +c_2((q^2+q)c_2-a_1q-2q-1)$ and hence $b_1= c_2((q^2+q)c_2 -a_1q -2q-1)$. As $b_1 = c_2(r-1)$ we find
$r-1 = (q^2+q)c_2 -a_1q -2q -1$, which in turn gives us $r = q((q+1)c_2 - a_1 -2)= -q(\theta_i+1)$.

This implies for $q >0$ we have $\theta_3 = -\frac{r}{q}-1$ and for $q \leqslant -1$ we have $\theta_1 = -\frac{r}{q}-1$, as $\theta_1 > -1 > \theta_3$.  Hence $G$ has three distinct $\mathbb{D}_q$-eigenvalues if and only if $\theta_3 = -\frac{r}{q}-1$ for $q >0$ or $\theta_1 = -\frac{r}{q}-1$ for $q \leqslant -1$. This shows the proposition.
\end{proof}

From Proposition \ref{antipodal d=3}, we have the following result for the distance matrix ($q=1$) of an antipodal distance-regular graph of diameter $3$.
\begin{corollary}
An antipodal distance-regular graph with diameter $3$ has exactly three distinct distance eigenvalues if and only if its smallest eigenvalue is $-r-1$.
\end{corollary}

\begin{remark}
\begin{itemize}
\item[(i)] Note that, Brouwer \cite{Brouwer1984} (see also \cite[pages 385-386]{BCN}) showed that, if a generalized quadrangle has parameters $(s,t)$ (where $t >1$), denoted by $GQ(s,t)$,  then the graph $G$ obtained by removing a spread, i.e. a set of vertex-disjoint $(s+1)$-cliques covering the vertices of the generalized quadrangle, from $GQ(s,t)$ is an antipodal distance-regular graph of diameter $3$ on $(s+1)(st+1)$ vertices with intersection array $(st, s(t-1), 1; 1, t-1, st)$ with smallest eigenvalue $-t$ and $k_3 = s$.
  So, if $s= t-2$, then $G$ is an antipodal $t-1$-cover with smallest eigenvalue $-t$, and hence $G$ has exactly 3 distinct distance eigenvalues. There are infinitely many $t$ for which a  $GQ(t-2, t)$
  has a spread, so this gives a new infinite family of distance-regular graphs with exactly 3 distinct distance eigenvalues. For more details see \cite[pages 385-386]{BCN}.

\item[(ii)] An antipodal distance-regular graph with intersection array
      $$\{(r+1)^2, (r-1)(r+2), 1; 1, r+2, (r+1)^2\}$$
      with $r \geq 2$ an integer, has $\theta_3 = -r-1$ and $k_3 = r-1$, and thus giving a distance-regular graph with exactly three distinct distance eigenvalues. An infinite family of these graphs was found by Alazemi et al. \cite{AAKS2017}.

\item[(iii)] The Taylor graphs with intersection array $\{3(c_2 -1), c_2,1; 1, c_2, 3(c_2 -1)\}$ with smallest eigenvalue $-3$, have exactly three distinct distance eigenvalues.
  There are exactly 4 such graphs, namely the $3$-cube, the Johnson graph $J(6,3)$, the halved 6-cube and the Gosset graph, see \cite[Corollary 1.15.3]{BCN}.
  \end{itemize}

\end{remark}
We can generalize Item (iii) as follows:
\begin{proposition}
For a fixed integer $r \geqslant 2$, there are finitely many antipodal $r$-cover distance-regular graphs of diameter $3$ with smallest eigenvalue $-r-1$.
\end{proposition}

\begin{proof}
Note that, for an antipodal $r$-cover distance-regular graph $G$ on $n$ vertices of diameter $3$ with distinct (adjacency) eigenvalues $k > \theta_1 > \theta_2 > \theta_3$, 
we have that $n =r(k+1)$, $k =-\theta_1 \theta_D$ and the multiplicity of $\theta_1$ is
\begin{equation*}
  m_1 = \frac{-\theta_3 (r-1)(k+1)}{\theta_1 -\theta_3} \text{ (see \cite[page 431]{BCN})}.
\end{equation*}
Now $\theta_3 = -r-1$ implies
\begin{equation*}
  m_1 = \frac{(r+1)(r-1)(\theta_1(r+1) +1)}{ \theta_1 +r +1}.
\end{equation*}
As $m_1 \in \mathds{Z}$, we know that $\theta_1 +r +1$ divides $(r+1)(r-1)(\theta_1(r+1) +1)$. This implies $\theta_1 +r +1$ divides $(r+1)(r-1)(\theta_1(r+1) -(r+1)(\theta_1 +r +1) +1) = -r(r+1)(r-1)(r+2)$. As $\theta_1 > 0$ and $r \geqslant 2$, we have $\theta_1 \leqslant (r+1)(r(r-1)(r+2)-1)$. Hence, $n =r(k+1) = r((r+1)\theta_1 +1) \leqslant  r^6 +3r^5 +r^4 -4r^3-4r^2$. This shows the proposition.
\end{proof}

\begin{remark}
Koolen and Park \cite{KP2010} found that, for distance-regular graphs of diameter $3$, if $\theta_1 = a_3$ and $b_2= c_2$, we have the intersection arrays
$$\{b^2(b-1)/2, (b-1)(b^2-b+2)/2, b(b-1)/4; 1, b(b-1)/4, b(b-1)^2/2\}.$$
These intersection arrays are feasible for positive integers $b \geqslant 2$ such that $b =0,1(\mod 4)$. If these graphs exist, they would have exactly $3$ distinct distance eigenvalues.
\end{remark}

\section{On distance eigenvalues ($q=1$) of antipodal distance-regular graphs of diameter $D$}
Note that a distance-regular graph with intersection array $\{b_0,b_1, \ldots, b_{D-1}; c_1, c_2, \ldots, c_D\}$ and diameter $D \in \{ 2d, 2d+1\}$ is antipodal if and only if $b_i =c_{D-1}$ for $i = 0, \ldots,D$, $i \neq d$. In this case, $G$ is an antipodal $r$-cover of its folded graph, where $r =1+ \frac{b_d}{c_{D-d}}$ (cf. \cite[Proposition 4.2.2. ]{BCN}). 

In 2015, Atik and Panigrahi \cite[Theorem 3.2.]{AP2015} showed that a distance-regular graph of diameter $D$ has at most $D+1$ distinct distance eigenvalue. Moreover, for distance-regular graphs of diameter $D$ satisfying $b_i =c_{D-i}$, for all $i =0, 1, \ldots, D$, they showed the following.

\begin{theorem}[Cf.{\cite[Theorem 3.3.]{AP2015}}]
Let $G$ be a distance-regular graph with diameter $D$ and satisfying $b_i =c_{D-i}$, for all $i =1,2, \ldots, D$. Then, zero is an eigenvalue of its distance matrix with multiplicity at least $\lfloor \frac{D}{2}\rfloor$ and so $G$ has at most $\lceil\frac{D}{2}\rceil +2$ distinct distance eigenvalues.
\end{theorem}

Here we give an alternative proof for this result.

\begin{proposition}
Let $G$ be an antipodal distance-regular graph of diameter $D$. If $r=2$, then zero is an eigenvalue of its distance matrix with multiplicity at least $\lfloor \frac{D}{2}\rfloor$ and so $G$ has at most $\lceil\frac{D}{2}\rceil +2$ distinct distance eigenvalues.
\end{proposition}

\begin{proof}
Let $G$ be an antipodal distance-regular graph with diameter $D \in \{2d,2d+1 \}$ and (adjacency) eigenvalues $k > \theta_1 > \theta_2 > \cdots > \theta_D$. The $q$-distance matrix of $G$ is
\begin{equation*}
  \mathbb{D}_q = A_1 +(1+\frac{1}{q})A_2 +(1 +\frac{1}{q} + \frac{1}{q^2})A_3 + \cdots + (1 +\frac{1}{q} + \cdots + \frac{1}{q^{D-1}})A_D.
\end{equation*}
As the standard sequence of the eigenvalue $\theta \neq k$ is $\textbf{u} =(1 \text{ } u_1 \text{ }  u_2 \text{ } \ldots \text{ } u_D)$, then if $q=1$, for $x \in V(G)$ we have that
\begin{equation*}
  (\mathbb{D}_1\textbf{u})_x = k\textbf{u}_1 +2k_2\textbf{u}_2 + \cdots + Dk_{D}\textbf{u}_D.
\end{equation*}

Since $\theta_2, \theta_4, \ldots, \theta_{2d}$ are the non-trivial eigenvalues of the folded graph of $G$, then $\textbf{u}_{D} =1$ and $\textbf{u}_i = \textbf{u}_{D-i}$ for $i \in \{ 1, 2, \ldots d \}$ (see the paragraph after Proposition 4.2.3 of \cite{BCN}). Thus,
\begin{equation*}
(\mathbb{D}_1\textbf{u})_x=\left\{
 \begin{array}{rl}
(2d+1)k_{2d+1} +(k + 2dk_{2d})\textbf{u}_1 + \cdots + (dk_{d} + (d+1)k_{d+1})\textbf{u}_d, & \text{if } D= 2d +1,\\
2dk_{2d}+(k +(2d-1)k_{2d-1})\textbf{u}_1 + \cdots + ((d-1)k_{d-1} + (d+1)k_{d+1})\textbf{u}_{d-1} + dk_d\textbf{u}_d, & \text{if } D= 2d.
 \end{array}
 \right.
\end{equation*}

If $r=2$, we have $k_D =1$ and $k_i = k_{D-i}$ for $i \in \{ 1, 2, \ldots, d\}$. Thus,
\begin{equation*}
(\mathbb{D}_1\textbf{u})_x=\left\{
 \begin{array}{rl}
(2d+1)(1 +k\textbf{u}_1 +k_2\textbf{u}_2 + \cdots + k_d\textbf{u}_d), & \text{if } D= 2d +1, \\
2d(1 +k\textbf{u}_1 + \cdots + k_{d-1}\textbf{u}_{d-1} + \frac{k_d}{2}\textbf{u}_d), & \text{if } D= 2d.
 \end{array}
 \right.
\end{equation*}
Let $\textbf{u}'_1 =\textbf{u}_1, \ldots, \textbf{u}'_{d-1} =\textbf{u}_{d-1}$, and $\textbf{u}'_d = \textbf{u}_d$ if $D= 2d +1$ and $\textbf{u}'_d = \frac{1}{2}\textbf{u}_d$ if $D= 2d +1$. Then the $\textbf{u}'_i$'s are orthogonal to the eigenvector of $\theta$, for $i \in \{1, 2, \ldots, d\}$ with $\theta \in \{\theta_2, \theta_4, \ldots, \theta_{2d}\}$. So the sum $1 + k_1\textbf{u}'_1 + \cdots + k_d\textbf{u}'_d $ is zero. This shows that $R_1(\theta_2) = R_1(\theta_4) = \cdots = R_1(\theta_{2d})= 0$ for $r=2$.  This shows that $0$ is an eigenvalue of the distance matrix with multiplicity at least $d = \lfloor \frac{D}{2}\rfloor$. Hence $G$ has at most $D +2 - \lfloor \frac{D}{2}\rfloor = \lceil\frac{D}{2}\rceil +2$ distinct distance eigenvalues.
\end{proof}

Besides the halved $(2d)$-cubes, the Johnson graphs $J(2D,D)$, the $D$-cubes, the doubled Odd graphs and the even polygons, the largest diameter of a known distance-regular antipodal 2-cover is 7, and is realized by the double coset graph of the binary Golay code, see \cite[Section 11.3E]{BCN}.


\section*{Declarations}

\subsection*{Declaration of Competing Interest}
The authors declare that there are no competing interests.

\subsection*{Acknowledgements}
J.H. Koolen is partially supported by the National Key R. and D. Program of China (No. 2020YFA0713100),
the National Natural Science Foundation of China (No. 12071454 and No. 12371339), and the Anhui Initiative in Quantum
Information Technologies (No. AHY150000).
M. Abdullah is supported by the Chinese Scholarship Council at USTC, China.
B. Gebremichel is supported by the National Key R. and D. Program of China (No. 2020YFA0713100) and
the Foreign Young Talents Program (No. QN2022200003L).
S. Hayat is supported by UBD Faculty Research Grants
(No. UBD/RSCH/1.4/FICBF(b)/2022/053).

\subsection*{Data availability}
No data was used for the research described in this article.


\end{document}